\documentclass[12pt,oneside]{amsart}
\usepackage[english]{babel} 
\usepackage{amsmath} 
\usepackage{amsthm}
\usepackage{amssymb}
\usepackage{geometry}
\usepackage{nicefrac} 
\usepackage{hyperref}
\usepackage{graphicx}
\usepackage{braket}
\usepackage{subfig}
\usepackage{csquotes}
\usepackage{color}
\providecommand{\abs}[1]{\left|#1\right|}
\providecommand{\norm}[1]{\left \| #1\right \|}

\newcommand{\I}{\mathcal{I}}
\newcommand{\JJ}{\mathcal{J}}
\newcommand{\K}{\mathcal{K}}
\newcommand{\R}{\mathbb{R}}
\newcommand{\N}{\mathbb{N}}

\numberwithin{equation}{section}

\theoremstyle{plain}
\newtheorem{theorem}{Theorem}[section]

\theoremstyle{plain}
\newtheorem{definition}[theorem]{Definition}

\theoremstyle{plain}
\newtheorem{lemma}[theorem]{Lemma}

\theoremstyle{plain}
\newtheorem{corollary}[theorem]{Corollary}

\theoremstyle{plain}
\newtheorem{proposition}[theorem]{Proposition}

\theoremstyle{remark}
\newtheorem{remark}[theorem]{Remark}

\theoremstyle{example}

\title[]{A family of nonlocal degenerate operators: maximum principles and related properties}
\date{}
\author{Delia Schiera}
\address{Departamento de Matemática do Instituto Superior Técnico, 
Universidade de Lisboa, 
Av. Rovisco Pais, 
1049-001 Lisboa, Portugal.}
\email{delia.schiera@tecnico.ulisboa.pt}
\begin{document}
\maketitle

\begin{abstract}
We consider a class of fully nonlinear nonlocal degenerate elliptic operators which are modeled on the fractional Laplacian and converge to the truncated Laplacians. We investigate the validity of (strong) maximum and minimum principles, and their relation with suitably defined principal eigenvalues. 
We also show a Hopf type Lemma, the existence of solutions for the corresponding Dirichlet problem, and representation formulas in some particular cases.
\end{abstract}

\section{Introduction}
Fix $N \in \mathbb{N}$, let $\Omega \subset \R^N$ be a bounded domain, and $u \in C^2(\Omega)$. 
In the last years, there has been an increasing interest in the so-called truncated Laplacian 
\begin{align} \nonumber \mathcal{P}_k^+(D^2 u)(x) &:= \sum_{i=N-k+1}^N \lambda_i(D^2u(x)) \\
&= \label{truncated} \max \left \{ \sum_{i=1}^k \langle D^2 u (x) \xi_i, \xi_i \rangle: \, \{ \xi_i\}_{i=1}^k \in \mathcal{V}_k \right\} \end{align}
where $1 \le k \le N$, $x \in \Omega$, $\lambda_i(D^2u(x)) $ are the eigenvalues of the Hessian of $u$ in the point $x$ in non-decreasing order, and 
$\mathcal{V}_k$  denotes the family of $k$-dimensional orthonormal sets on $\R^N$. Similarly one defines $\mathcal{P}_k^-$ as the sum of the first $k$ eigenvalues, or equivalently substituting the maximum in \eqref{truncated} with a minimum. These operators are nonlinear and degenerate elliptic, in the sense that, given $X$, $Y$ two $N \times N$ real symmetric matrices, 
\[ \langle X \xi, \xi \rangle \le \langle Y \xi, \xi \rangle \text{ for any } \xi \in \R^N \quad \Rightarrow \quad \mathcal{P}_k^\pm(X) \le \mathcal{P}_k^\pm(Y). \]
We refer the reader to \cite{CaffarelliLiNirenberg, HarveyLawson, BGI, BGI1} for some insights on the truncated Laplacians, and further references. 

In \cite{BGT} a new class of nonlocal degenerate elliptic fully nonlinear operators was introduced, which provides a suitable nonlocal analog to the truncated Laplacian. Precisely, we set for $u \in L^\infty(\R^N) \cap C^2(\Omega)$, $\xi \in \R^N$ and $s \in (0, 1)$ 
\[  \mathcal{I}_\xi u(x) :=C_s \int_{0}^{+\infty} \frac{u(x+\tau \xi) -u(x)}{\tau^{1+2s}} \, d\tau, \]
where $C_s$ is a suitable normalizing constant. 
These operators act like a $2s$ derivative in the direction $\xi$, and they are built upon the definition of fractional Laplacian
\[ -(-\Delta)^s u(x):= \frac 12 C_{N, s} \int_{\R^N} \frac{u(x+y) + u(x-y)-2u(x)}{|y|^{N+2s}} \, dy, \]
for which we refer for instance to \cite{BV}. 
Then, one substitutes the second order derivatives $\langle D^2 u (x) \xi_i, \xi_i \rangle$ appearing in \eqref{truncated} with $\I_{\xi_i}$, to get the following definitions
\[ \I_k^+ u(x) := \sup \left \{ \sum_{i=1}^k \mathcal{I}_{\xi_i} u(x) \colon \{ \xi_i \}_{i=1}^k \in \mathcal{V}_k \right \} \]
\[ \I_k^- u(x) := \inf \left \{ \sum_{i=1}^k \mathcal{I}_{\xi_i} u(x) \colon \{ \xi_i \}_{i=1}^k \in \mathcal{V}_k \right \}. \]
We recall that in the particular case $k=1$, they were also considered in the paper \cite{DelPezzoQuaasRossi}. 
As a first remark, we notice that the supremum (infimum) in the definition above in general is not attained, see \cite{BGS}, whereas it is in the case of the truncated Laplacian.
This is related to another peculiar property of $\I_k^\pm$, for which we also refer to \cite{BGS}, namely the fact that the map $x \mapsto \I_k^\pm u(x)$ is not continuous (even for $u \in C^\infty(\Omega)$), differently from what happens for the local case, and also for the integro-differential operators which were taken into account in \cite{CaffarelliSilvestre}. 

Also, if the normalizing constant $C_s$ is chosen in a suitable way, $\I_k^\pm$ converge to the truncated Laplacian as $s \to 1$, see \cite{BGT}. 
Moreover, we stress that even if $k=N$ these operators do not coincide with the fractional Laplacian, and furthermore $\I_N^+ \ne \I_N^-$. 
In \cite{BGT, BGS} many properties of $\I_k^\pm$ have been proved, in particular, representation formulas, Liouville type theorems, comparison and maximum principles, existence of a principal eigenfunction in the case $k=1$, and some properties of principal eigenvalues. 

Now, a natural question is whether it is possible to consider a class of operators with nonlocal diffusion in a $n$-dimensional space, with $1 \le n \le N$. Having this in mind, following \cite{BGT}, define
\[ \mathcal{J}_V u(x):= C_{n, s} P.V. \int_{V} \frac{ u (x+z) - u(x)}{\abs{z}^{n+2s}} \, d \mathcal{H}^n(z)\]
where $P.V.$ stands for principal value, $V$ is a $n$-dimensional subspace of $\R^N$,  $\mathcal{H}^n$ is the $n$-dimensional Hausdorff measure in $\R^N$, and $C_{n, s}$ is a normalizing constant. If $V=\langle \xi_1, \dots, \xi_n \rangle$ then
\[ \mathcal{J}_V u(x)= C_{n, s} P.V. \int_{\R^n} \frac{ u (x+\sum_{i=1}^n \tau_i \xi_i) - u(x)}{(\sum_{i=1}^n \tau_i^2)^{\frac{n+2s}{2}}} \, d\tau_1 \dots d\tau_n. \]
One can now define
\[ \JJ_k^+ u(x):= \sup \left\{ \JJ_V u(x) \colon V= \langle \xi_1, \dots, \xi_k \rangle, \{ \xi_i \}_{i=1}^k \in \mathcal{V}_k \right\}\]
\[ \JJ_k^- u(x):= \inf \left\{ \JJ_V u(x) \colon V= \langle \xi_1, \dots, \xi_k \rangle, \{ \xi_i \}_{i=1}^k \in \mathcal{V}_k \right\}.\]
Also for these operators, representation formulas and Liouville type results have been shown in \cite{BGT}. Moreover, one can choose the normalization constants $C_{n, s}$ such that $\JJ_k^\pm$ converge to the truncated Laplacian as $s \to 1$. 

Our aim is to further extend and complement this theory, considering suitable ``mixed" operators which converge to the truncated Laplacian as $s \to 1$, and include as special cases $\I_k^\pm $ and $\JJ_k^\pm$, see Section \ref{sec:prelimin} below for the precise definition. In order to give a rough idea of the operators we have in mind, a toy model is given by the following operator
\[ \sup_{\xi} \sup_{V \perp \xi} \left( \I_{\xi}u(x) + \JJ_V u(x) \right), \]
where $V$ is, for instance, a $2$-dimensional space in $\R^N$. More generally, instead of summing the $1$-dimensionally nonlocal operators $\I_\xi$ and taking the supremum (infimum), as we did in defining $\I_k^\pm$, here we will consider sums of $\JJ_V$ of (possibly) different dimensions.

This paper is intended as an extension of the results in \cite{BGS, BGT} to these operators. We stress that in most cases, the results presented here are new even for the operators $\JJ_k^\pm$, $1 < k < N$. 
In particular we will show that also this class of general nonlocal degenerate operators lacks continuity, see Section \ref{sec:continuity}, and we study validity of comparison and maximum principles in Section \ref{sec:maximum}. Section \ref{sec:hopf} is devoted to the proof of a Hopf type lemma, whereas in Section \ref{sec:dirichlet} we focus on existence for the corresponding Dirichlet problem, and on some properties of the principal eigenvalues. 
We conclude remarking that representation formulas for this class of operators are not trivial to get, and only some particular cases can be taken into account, see Appendix \ref{sec:repres}.

\section{Definition and preliminaries}\label{sec:prelimin}

Choose $1\le \ell\le N$. Let $1 \le k_1 \le \dots \le k_\ell \le N$ such that 
\[ k:=\sum_{j=1}^\ell k_j \]
 satisfies $1 \le k \le N$. 
Let us denote 
\begin{align*} \nonumber {\sup}^{k_1, \dots, k_\ell}  &:= \sup_{\{ \xi_j^1 \}_{j=1}^{k_1} \in  \mathcal{V}_{k_1}} \sup_{\{ \xi_j^2 \}_{j=1}^{k_2} \in  \mathcal{V}_{k_1, k_2} } \dots \sup_{\{ \xi_j^\ell \}_{j=1}^{k_\ell} \in  \mathcal{V}_{k_1, \dots, k_\ell} } \\
 & = \sup \Big \{ \{ \xi_j^1 \}_{j=1}^{k_1} \in  \mathcal{V}_{k_1}, \, \{ \xi_j^2 \}_{j=1}^{k_2} \in  \mathcal{V}_{k_1, k_2}, \, \dots, \, \{ \xi_j^\ell \}_{j=1}^{k_\ell} \in  \mathcal{V}_{k_1, \dots, k_\ell} \Big \}, 
 \end{align*}
where $\mathcal{V}_{k_1}$ is the collection of all $k_1$-orthonormal sets of $\R^N$, and 
$ \mathcal{V}_{k_1, \dots, k_t} $, $t \ge 2$, represents the collection of all $k_t$-orthonormal sets of $\R^N$ which are orthogonal to the space generated by the vectors $\xi_j^s$, with $j=1, \dots, k_s$ and $s=1, \dots, t-1$. 
Let also 
\begin{equation}\label{defV} V_i:=\langle \xi_j^i \rangle_{j=1}^{k_i}. \end{equation}
Then for $u \in L^\infty (\R^N) \cap C^2(\Omega)$ we define 
\[ \K_{k_1, \dots, k_\ell}^+ u(x):= {\sup}^{k_1, \dots, k_\ell}  \sum_{i=1}^\ell \JJ_{V_i} u(x). \]
Analogously, we define $\K_{k_1, \dots, k_\ell}^-$ taking the infimum in place of the supremum.

\begin{remark}
Notice that  
\[ \K_{k_1, \dots, k_\ell}^+ u(x)= \sup_{\{ \xi_j \}_{j=1}^k \in  \tilde{\mathcal{V}}_k}  \sum_{i=1}^\ell \JJ_{\tilde V_i} u(x), \]
where 
\[ \tilde V_i= \langle \xi_j \rangle, \text{ with } j \in \left\{ \sum_{t=1}^{i-1} k_t +1, \dots, \sum_{t=1}^{i} k_t \right\} \]
and $\tilde{\mathcal{V}}_k$ is the set of all \textit{ordered} $k$-uples of orthonormal vectors in $\R^N$.

\end{remark}

\begin{remark}
By definition 
\[ {\sup}^{k_1}  =  \sup_{\{ \xi_j^1 \}_{j=1}^{k_1} \in  \mathcal{V}_{k_1}}. \]
Thus, if $\ell=1$, then $k=k_1$, and $\K_{k}^{\pm} = \JJ_k^\pm$. In particular, if $\ell=1$, and $k_1=1$, then $\K_1^{\pm} =\JJ_1^\pm=\I_1^\pm$, whereas if $\ell=1$, and 
$k_1=N$, then $\K_{N}^{\pm} =\JJ_N^\pm=-(-\Delta)^s$. 

Also, if $k_1= \dots = k_\ell=1$, then $k=\ell$ and 
\[ {\sup}^{1, \dots, 1}\sum_{i=1}^\ell \JJ_{V_i} u(x) = \sup_{\{ \xi_j \}_{j=1}^{k} \in  \mathcal{V}_{k}} \sum_{i=1}^\ell \I_{\xi_i } u(x), \]
so that $\K_{1, \dots, 1}^\pm = \I_k^\pm$. 
In particular, if $\ell=N$, then $\K_{1, \dots, 1}^\pm=\I_N^\pm$.

We finally notice that if $\ell \ne 1$, then $\K_{k_1, \dots, k_\ell}^\pm $ does not coincide with the fractional Laplacian, even if $k=N$. 
\end{remark}

\begin{remark}\label{rmk:normalizing constant}
We point out that if we choose the normalization constants $C_{k_i,s}$ such that 
\begin{equation}\label{asymptotic} \frac{C_{k_i, s} |\mathcal{S}^{k_i-1}|}{4k_i(1-s)} \to 1 \quad \text{ as } s \to 1^-, \end{equation}
then
\[ \K_{k_1, \dots, k_\ell}^\pm u(x) \to \mathcal{P}_k^\pm u(x) \quad \text{ as } s \to 1^-, \]
see also \cite[Lemma 6.1]{BGT}. An explicit formula for these constants can be found in \cite[Equations (3.1.10) and (3.1.15)]{BV}.
\end{remark}

Less regularity on the function $u$ can be imposed, once we deal with the notion of viscosity solutions, see \cite{BCI, BarlesImbert}. For definitions and main properties of viscosity solutions in the classical local framework we refer to the survey \cite{CIL}. 

\begin{definition}  
Let $f \in C(\Omega \times \R)$. We say that $u \in L^\infty(\R^N) \cap LSC(\Omega) $ (respectively $USC(\Omega)$) is a (viscosity) supersolution (respectively subsolution) to 
\begin{equation}\label{eq def sol} \K_{k_1, \dots, k_\ell}^+ u +f(x, u(x)) = 0 \text{ in } \Omega \end{equation}
if for every point $x_0 \in \Omega$ and every function $\varphi \in C^2(B_\rho(x_0))$, $\rho >0$, such that $x_0$ is a minimum (resp. maximum) point to $u - \varphi$, one has  
\begin{equation*}\label{cond super} \mathcal{K}(u, \varphi, x_0, \rho) +f(x_0, u(x_0)) \le 0 \quad \text{(resp. $\ge 0$)}\end{equation*}
where 
\begin{align*} \mathcal{K}(u, \varphi, x_0, \rho) = {\sup}^{k_1, \dots, k_\ell}  \sum_{i=1}^\ell  &C_{k_i,s}\bigg\{\int_{B_\rho(0)}\frac{\varphi(x_0+\sum_{j=1}^{k_i} \tau_j \xi_j^i) - \varphi(x_0)}{(\sum_{j=1}^{k_i} \tau_j^2)^{\frac{k_i+2s}{2}}} \, d\tau_1 \dots d \tau_{k_i}  \\
&+ \int_{B_\rho(0)^c} \frac{u(x_0+\sum_{j=1}^{k_i} \tau_j \xi_j^i)-u(x_0)}{(\sum_{j=1}^{k_i} \tau_j^2)^{\frac{k_i+2s}{2}}} \, d\tau_1 \dots d \tau_{k_i}\bigg\}.  \end{align*}
 We say that a continuous function $u$ is a  solution of \eqref{eq def sol} if it is both a supersolution and a subsolution of \eqref{eq def sol}.
We analogously define viscosity sub/super solutions for the operator $\K_{k_1, \dots, k_\ell}^-$, taking the infimum   in place of the supremum. 
\end{definition}

\section{Lack of continuity and related phenomena}\label{sec:continuity}
In this section  we study continuity properties of the maps $x \mapsto \K_{k_1, \dots, k_\ell}^\pm u(x)$. As for the case of the operator $\I_k^\pm$, see \cite{BGS}, assuming $u \in C^2(\Omega) \cap L^\infty(\R^N)$ is in general not enough to guarantee the continuity of $\K_{k_1, \dots, k_\ell}^\pm u(x)$ with respect to $x$. 

%%%%
Assume that $\ell \ne 1$, or $\ell=1$ and $k_1< N$. The case $\ell=1$, $k_1=N$ reduces to the fractional Laplacian  (which is continuous). 
Consider the function 
\begin{equation*}
u(x)=\left\{\begin{array}{rl}
0 &  \text{if $|x|\leq1$, or $\exists i=1, \dots, \ell$ s.t. $x \in \langle e_j \rangle_{j \in \mathcal{A}_i}$}\\
-1 & \text{otherwise,}
\end{array}\right.
\end{equation*}
where 
\begin{equation}\label{Ai} \mathcal{A}_i=\left \{\sum_{j=1}^{i-1} k_j +1, \dots, \sum_{j=1}^{i} k_j \right \}, \, i=1, \dots, \ell. \end{equation}
Notice that $u \equiv 0$ if and only if $\ell=1$ and $k_1=N$.
Set $\Omega=B_1(0)$. The map
$$
x\in\Omega\mapsto\K_{k_1, \dots, k_\ell}^+ u(x)
$$
is well defined, since $u$ is bounded in $\R^N$ and smooth in $\Omega$. We shall prove that it is not continuous at $x=0$.

Let us first compute the value $\K_{k_1, \dots, k_\ell}^+ u(0)$. Since $u\leq0$ in $\R^N$ it turns out that 
\begin{equation*}
\JJ_{V_{i}} u(0)=C_{k_i, s} \int_{\R^{k_i}} \frac{u(\sum_{j=1}^{k_i} \tau_j \xi_{j}^i )}{(\sum_{j=1}^{k_i} \tau_j^2)^{\frac{k_i+2s}{2}}}\,d\tau_1 \dots d\tau_{k_i} \leq0, \quad i=1, \dots, \ell,
\end{equation*}
where $V_{i}$ is defined in \eqref{defV}, for any choice of $\xi_j^i$, $j=1, \dots, k_i$. 
Hence
\begin{equation}\label{eq1}
{\sup}^{k_1, \dots, k_\ell}   \sum_{i=1}^\ell\JJ_{V_i}  u(0)\leq0.
\end{equation}
On the other hand, choosing the vectors $e_1,\ldots,e_k$ of the standard basis, we obtain that 
\begin{equation}\label{eq2}
\JJ_{\langle e_j \rangle _{j \in \mathcal{A}_i}} u(0)=0, \quad i=1, \dots, \ell.
 \end{equation}
Hence by \eqref{eq1}-\eqref{eq2} 
\[
\K_{k_1, \dots, k_\ell}^+ u(0)=0.
 \]

Now we are going to prove that 
$$
\limsup_{n\to+\infty}\K_{k_1, \dots, k_\ell}^+ u\left(\frac1ne_N\right)<0
$$
where $e_N=(0,\ldots,0,1)$. 

We preliminary notice that $u(\frac1n e_N)=0$ as $\frac1n e_N \in  B_1$. 
We first consider $\ell=1$, $k_1 < N$. Then 
the space $V_{1}+\frac1n e_N$ does not coincide with  the space $\langle e_j \rangle_{j \in \mathcal{A}_1}= \langle e_1, \dots, e_{k_1} \rangle$, since $k_1 < N$. 
Thus, also using $u \le 0$, 
\[
\begin{split}
\JJ_{V_1} u\left(\frac1n e_N\right)&=C_{k_1, s} \int_{V_1 }\frac{u(\frac1ne_N+z)}{\abs{z}^{k_1+2s}}\,d \mathcal{H}^{k_1}(z)\\
&\le - C_{k_1, s}\int_{V_1 \cap \{ z: \abs{z} \ge \tau_1(z, n), \langle z, e_N \rangle>0 \} } \frac{1}{\abs{z}^{k_1+2s}}\,d \mathcal{H}^{k_1}(z)
\end{split}
\]
where 
\[ \tau_1(z, n)=-\frac{\left\langle \hat z,e_N\right\rangle}{n}+\sqrt{\left(\frac{\left\langle \hat z,e_N\right\rangle}{n}\right)^2+1-\frac{1}{n^2}}, \, \quad \qquad \hat z=\frac{z}{|z|}. \]
Notice that
\[  \tau_1(z, n) \le \sqrt{1-\frac{1}{n^2}}\]
hence
\[\begin{split} \JJ_{V_1} u\left(\frac1n e_N\right) &\le - C_{k_1, s}\int_{V_1\cap \{ z: \abs{z} \ge \sqrt{1-\frac{1}{n^2}}, \langle z, e_N \rangle >0 \} } \frac{1}{\abs{z}^{k_1+2s}}\,d \mathcal{H}^{k_1}(z)\\
&= - \frac 12 C_{k_1, s}\int_{\sqrt{1-\frac{1}{n^2}}}^{+\infty} \frac{1}{r^{k_1+2s}} r^{k_1-1} \, dr=- \frac{C_{k_1, s}}{4s(1-\frac{1}{n^2})^{s}}.
\end{split} \]
Thus 
\[
\limsup_{n\to+\infty}\K_{k_1}^+ u\left(\frac1n e_N\right)\leq -\frac{1}{4s} C_{k_1, s}<0.
\]
Let us now consider the case $\ell>1$. For each space $V_i$ two situations may occur:
\begin{enumerate}
\item $V_i + \frac1n e_N$ is contained in $\langle e_j \rangle_{j \in \mathcal{A}_l}$ for some $l$, and in this case $\JJ_{V_i} u=0$.
\item $V_i + \frac1n e_N$ is not contained in $\langle e_j \rangle_{j \in \mathcal{A}_l}$ for any $l$, and in this case we can perform the same computations above to get 
\[ \JJ_{V_i} u\left(\frac1n e_N\right) \le - \frac{C_{k_i, s}}{4s(1-\frac{1}{n^2})^{s}}.
 \]
\end{enumerate}
Notice that there exists at most one couple of indices $(\hat i, l)$ such that (1) occurs, as (1) implies that $e_N \in V_{\hat i}$, and $V_j$ is orthogonal to $V_{\hat i}$ if $j \ne \hat i$. 
Thus, 
\[
\limsup_{n\to+\infty}\K_{k_1, \dots, k_\ell}^+ u\left(\frac1n e_N\right)\leq -\frac{1}{4s} \min_{j=1, \dots \ell} \sum_{\substack{i=1\\i \ne j}}^{\ell} C_{k_{i}, s}<0,
\]
and we get the conclusion. 

The counterexample above shows that continuity in general does not hold. However, adding a global assumption, we recover (semi) continuity of the operators. The proof is completely analogous to \cite[Proposition 3.1]{BGS}. 
\begin{proposition}\label{semicont}
Let $u\in C^2(\Omega)\cap L^\infty(\R^N)$, and consider the map
\begin{equation*}
\K_{k_1, \dots, k_\ell}^\pm u:x\in\Omega\mapsto \K_{k_1, \dots, k_\ell}^\pm u(x).
\end{equation*}
If $u\in LSC(\R^N)$ (respectively $USC(\R^N)$, $C(\R^N)$) then
$\K^\pm_{k_1, \dots, k_\ell} u\in LSC(\Omega)$ (respectively $USC(\Omega)$, $C(\Omega)$).
\end{proposition}

In what follows, we give some counterexamples to prove that the $\sup$ or $\inf$ in the definition of $\K_{k_1, \dots, k_\ell}^\pm$ are in general not attained under the only assumption $u \in C^2(\Omega) \cap L^\infty(\R^N)$, see \cite{BGS} for the case of the operator $\I_k^\pm$. 

%%%%%%%%%
Let us first consider the case $\ell=1$ and $k_1=k < N$ (recall that the case $\ell=1$, $k=N$ corresponds to the fractional Laplacian). 
Take the function
\[ u(x)= \begin{cases}
e^{-\langle x, e_N \rangle} &\text{ if } \langle x, e_N \rangle >0 \text{ and } \abs{x} >1\\
0 &\text{ otherwise.}
\end{cases}\]
Let $V$ generated by $\xi_i$, $i=1, \dots, k$, not contained in $\langle x, e_N \rangle =0$. One has
\begin{align*} \JJ_V u(0)&=\int_V \frac{u(x)}{\abs{x}^{2s+k}} \,d \mathcal{H}^{k}(x) = \int_{V \cap B_1^c \cap \{ \langle x, e_N \rangle >0 \}} \frac{e^{-\langle x, e_N \rangle} }{\abs{x}^{2s+k}} \,d \mathcal{H}^{k}(x) \\
&= \int_{\{ \abs{\tau}>1, \sum_i \tau_i \langle \xi_i, e_N \rangle>0 \}} \frac{e^{-\sum_i \tau_i \langle \xi_i, e_N \rangle }}{\abs{\tau}^{2s+k}} \, d \tau_1 \dots d \tau_k. \end{align*}
Notice that 
\[ \int_{\{ \abs{\tau}>1, \sum_i \tau_i \langle \xi_i, e_N \rangle>0 \}} \frac{e^{-\sum_i \tau_i \langle \xi_i, e_N \rangle }}{\abs{\tau}^{2s+k}} \, d \tau_1 \dots d \tau_k = \int_{\{ \abs{\tau}>1, \tau_1>0 \}} \frac{e^{-\tau_1 \langle \tilde \xi_1, e_N \rangle }}{\abs{\tau}^{2s+k}} \, d \tau_1 \dots d \tau_k \]
up to choosing a different basis for $V$ such that $\langle \tilde \xi_i, e_N \rangle=0$ for $i \ge 2$, and $\langle \tilde \xi_1, e_N \rangle>0$.
Now, calling
\[ f(y)=\int_{\{ |\tau|>1, \tau_1 >0 \}}  \frac{e^{-\tau_1 y}}{\abs{\tau}^{2s+k}} \, d \tau_1 \dots d \tau_k, \]
we have 
\[ \sup_{y \in (0, 1]} f(y)=f(0)=\begin{cases}
\frac12 \int_1^\infty \frac{1}{r^{2s+1}} \, dr =\frac 1{4s}  &\text{ if } k>1\\
\int_1^\infty \frac{1}{r^{2s+1}} \, dr =\frac 1{2s}  &\text{ if } k=1. 
\end{cases} \]
Hence 
\[ \K_{k}^+ u(0)=\begin{cases}
\frac 1{4s}&\text{ if } k>1\\
\frac 1{2s} &\text{ if } k=1. \end{cases} \]
However, this is not attained, as $\JJ_Vu(0)=0$ by definition of $u$, if  $V$  is contained in $\langle x, e_N \rangle =0$. 

We now consider the case $\ell \ge 2$. If $k_\ell=1$, we refer to \cite{BGS}. A counterexample for the general case is the following 
\[ u(x)= \begin{cases}
e^{-\langle x, e_N \rangle} &\text{ if } \sum_{i=1}^{N-k_1-1} \langle x, e_i \rangle^2=0, \langle x, e_N \rangle >0 \text{ and }\abs{x}>1\\
0 &\text{ otherwise}.
\end{cases}\]
Notice that $\JJ_V u(0)=0$ for every $V$ of dimension $> k_1+1$, as $u \ne 0$ on a space of dimension $k_1+1$, and $=0$ everywhere else. Therefore these elements never give a significant contribution to the supremum. 
As for the spaces $V$ of dimension $\le k_1+1$, we need to split the analysis into two cases.

Let us first assume $k_1=k_2=1$. Fix any orthonormal set of $\R^k$. Thus, there are at most two vectors in this set which belong to the space $ \sum_{i=1}^{N-k_1-1} \langle x, e_i \rangle^2=0 $. We distinguish three possible situations:
\begin{itemize}
\item[(i)] If they both belong to a space $V_j$ of dimension $k_1+1=2$ (if one such space exists), then this is the only space which makes a significant contribution, as $\JJ_{V_i}u(0)=0$ for $i \ne j$. We then use the case $\ell=1$ to conclude that 
\begin{equation}\label{bound 4s} \sum_{i=1}^\ell \JJ_{V_i} u(0) = \JJ_{V_j} u(0) \le \frac{1}{4s}. \end{equation}
\item[(ii)] If one of the two vectors belongs to a space of dimension $k_1=1$, and the other one to a space of dimension $>1$, then 
\begin{equation}\label{bound 2s} \sum_{i=1}^\ell \JJ_{V_i} u(0) \le \frac{1}{2s}, \end{equation}
again exploiting computations above. 
\item[(iii)] Finally, if they belong to two different spaces of dimension $k_1=1$, then we have a similar situation as in \cite{BGS}, and  
\[ \sum_{i=1}^\ell \JJ_{V_i} u(0) \le \int_{1}^\infty \frac{1+ e^{-\tau}}{\tau^{2s+1}} \, d\tau. \]
\end{itemize}
Then we conclude that for any orthonormal set of $\R^k$ 
\[ \sum_{i=1}^\ell \JJ_{V_i} u(0) \le \max \left \{ \frac{1}{4s}, \frac{1}{2s}, \int_1^\infty \frac{1+ e^{-\tau}}{\tau^{2s+1}} \, d\tau \right \} = \int_1^\infty \frac{1+ e^{-\tau}}{\tau^{2s+1}} \, d\tau. \]
From this we deduce
\[ \K_{k_1, \dots, k_\ell}^+ u(0)= \int_1^\infty \frac{1+ e^{-\tau}}{|\tau|^{2s+1}} \, d\tau, \]
as we can find a sequence which converges to this value, arguing as in \cite{BGS}. 
However, this is not attained.

We now take into account the situation in which $1=k_1<k_2$, or $k_1>1$. Again, for any orthonormal set of $\R^k$ there are at most $k_1+1$ vectors such that 
$\sum_{i=1}^{N-k_1-1} \langle x, e_i \rangle^2=0$. If they all belong to one of the spaces of dimension $k_1+1$ (if it exists), then \eqref{bound 4s} holds. If $k_1$ of them appear in a space of dimension $k_1$, then \eqref{bound 4s} holds if $k_1>1$, whereas one has \eqref{bound 2s} if $k_1=1$. In any other case, the sum is identically 0. 
Thus, by similar arguments as above, 
\[ \K_{k_1, \dots, k_\ell}^+ u(0)= 
\begin{cases}
\frac{1}{4s} &\text{ if } k_1>1 \\
\frac{1}{2s} &\text{ if } k_1=1, 
\end{cases} \]
and the supremum is not attained.

%%%%%%%%%%%%%%

\section{Comparison and maximum principles}\label{sec:maximum}

We consider the problems
\begin{equation}\label{2505eq2}
\left\{\begin{array}{cl}
\K^\pm_{k_1, \dots, k_\ell} u+c(x)u=f(x) &  \text{in $\Omega$}\\
u=0 & \text{in $\R^N\backslash\Omega$}
\end{array}\right.
\end{equation}
and we prove a comparison principle. 
\begin{theorem}\label{comparison}
Let $\Omega\subset\R^N$ be a bounded domain and let $c(x),f(x)\in C(\Omega)$ be such that $\left\|c^+\right\|_{\infty}<\sum_{i=1}^\ell   C_{k_i, s}  \frac {1} {2s}(\text{diam}(\Omega))^{-2s}$. If $u\in USC(\overline\Omega)\cap L^\infty(\R^N)$ and $v\in LSC(\overline\Omega)\cap L^\infty(\R^N)$ are respectively sub and supersolution of \eqref{2505eq2}, then $u\leq v$ in  $\Omega$.
\end{theorem}
\begin{proof}
We show the proof only in the case $\K^+_{k_1, \dots, k_\ell}$, as the same arguments apply to $\K^-_{k_1, \dots, k_\ell}$ as well.
We argue by contradiction by supposing that there exists $x_0\in\Omega$ such that
\[
\max_{\R^N}(u-v)=u(x_0)-v(x_0)>0.
\]
Doubling the variables,  for $n\in\mathbb N$ we consider  $(x_n,y_n)\in\overline\Omega\times\overline\Omega$ such that 
\begin{equation}\label{2505eq3}
\max_{\overline\Omega\times\overline\Omega}(u(x)-v(y)-n|x-y|^2)=u(x_n)-v(y_n)-n|x_n-y_n|^2\geq u(x_0)-v(x_0).
\end{equation}
Using \cite[Lemma 3.1]{CIL}, up to subsequences, we have
\begin{equation}\label{2505eq6}
\lim_{n\to+\infty}(x_n,y_n)=(\bar x,\bar x)\in\Omega\times\Omega
\end{equation}
and 
\begin{equation}\label{2505eq7}
\lim_{n\to+\infty}u(x_n)=u(\bar x),\quad \lim_{n\to+\infty}v(x_n)=v(\bar x),\quad u(\bar x)-v(\bar x)=u(x_0)-v(x_0).
\end{equation}
We know that for $n\geq\frac{\left\|u\right\|_\infty+\left\|v\right\|_\infty}{\varepsilon^2}$  one has 
\begin{equation}\label{2505eq5}
\max_{\overline\Omega\times\overline\Omega}[ u(x)-v(y)-n|x-y|^2]=\max_{\R^N\times\R^N}[u(x)-v(y)-n|x-y|^2]\,,
\end{equation}
see for instance \cite[equation (4.7)]{BGS}. 

 Taking $\varphi_n(x):=u(x_n)+ n|x-y_n|^2 - n|x_n-y_n|^2$ and $\phi_n(y)=v(y_n)- n|x_n-y|^2 +n|x_n-y_n|^2$, we see that $\varphi_n$ touches $u$ in $x_n$ from above, while $\phi_n$ touches $v$ in $y_n$ from below. 
Let us define in order to simplify the notation
\[ \delta(w, z, \eta)=w(z+\eta)-w(z), \]
and $\tau=(\tau_1, \dots, \tau_k)$, $d\tau = d\tau_1\dots d\tau_k$.
 
 Hence we obtain (see also the proof of \cite[Theorem 4.1]{BGS}), 
\begin{equation}\label{2505eq10}
\begin{split}
f(x_n)&-f(y_n)\leq \sum_{i=1}^\ell   C_{k_i, s}  \frac{2 n\rho^{2-2s}}{1-s}+c(x_n)u(x_n)-c(y_n)v(y_n)
\\ &\hspace{1cm} +{\sup}^{k_1, \dots, k_\ell} \sum_{i=1}^\ell   C_{k_i, s} \int_{B_\rho(0)^c}  \frac{\delta(u, x_n,  \sum_{j=1}^{k_i} \tau_j \xi_j^i ) - \delta(v, y_n,  \sum_{j=1}^{k_i} \tau_j \xi_j^i )}{\abs{\tau}^{k_i+2s}}\,d\tau.
\end{split}
\end{equation}
From \eqref{2505eq3} and \eqref{2505eq5}  we have
$$
u(x)-v(y)-n|x-y|^2\leq u(x_n)-v(y_n)-n|x_n-y_n|^2\quad\forall x,y\in\R^N.
$$
Choosing in particular $x=x_n+\sum_{j=1}^{k_i} \tau_j \xi_j^i$ and $y=y_n+\sum_{j=1}^{k_i} \tau_j \xi_j^i$ we deduce that 
\[ \delta(u, x_n,  \sum_{j=1}^{k_i} \tau_j \xi_j^i) - \delta(v, y_n,  \sum_{j=1}^{k_i} \tau_j \xi_j^i) \le 0. \]
 Thus \eqref{2505eq10} implies, assuming without loss of generality that $\rho < \text{diam}(\Omega)$, 
\begin{equation}\label{2505eq11}
\begin{split}
f(x_n)-&f(y_n)\leq \sum_{i=1}^\ell   C_{k_i, s}  \frac{2 n\rho^{2-2s}}{1-s}+c(x_n)u(x_n)-c(y_n)v(y_n)\\
&\hspace{0.5cm}+ {\sup}^{k_1, \dots, k_\ell}\sum_{i=1}^\ell  C_{k_i, s} \int_{B_{\text{diam}(\Omega)} (0)^c}  \frac{\delta(u, x_n,  \sum_{j=1}^{k_i} \tau_j \xi_j^i ) - \delta(v, y_n,  \sum_{j=1}^{k_i} \tau_j \xi_j^i )}{\abs{\tau}^{k_i+2s}}\,d\tau.
\end{split}
\end{equation}
Since $\Omega\subset B_{\text{diam}(\Omega)}(x_n)$ and $x_n+\sum_{j=1}^{k_i} \tau_j \xi_j^i\notin B_{\text{diam}(\Omega)}(x_n)$ for any $\abs{\tau}\geq\text{diam}(\Omega)$, then $u(x_n+\sum_{j=1}^{k_i} \tau_j \xi_j^i)\leq0$. For the same reason $v(y_n+ \sum_{j=1}^{k_i} \tau_j \xi_j^i)\ge0$ when $\abs{\tau}\geq\text{diam}(\Omega)$. Hence 
$$
\delta(u, x_n,  \sum_{j=1}^{k_i} \tau_j \xi_j^i) - \delta(v, y_n,  \sum_{j=1}^{k_i} \tau_j \xi_j^i)\leq -(u(x_n)-v(y_n))
$$ 
and 
\begin{equation}\label{2505eq12}
\begin{split}
f(x_n)-f(y_n)\leq & \; \sum_{i=1}^\ell   C_{k_i, s}  \frac{2n\rho^{2-2s}}{1-s}+c(x_n)u(x_n)-c(y_n)v(y_n)\\
&\quad -\sum_{i=1}^\ell   C_{k_i, s}  \frac {1} {2s}(\text{diam}(\Omega))^{-2s} (u(x_n)-v(y_n)).
\end{split}
\end{equation}
Letting first $\rho\to0$, then $n\to+\infty$ and using \eqref{2505eq6}-\eqref{2505eq7} we obtain
$$
0\leq (u(x_0)-v(x_0))\left(c(\bar x)-\sum_{i=1}^\ell   C_{k_i, s}  \frac {1} {2s}(\text{diam}(\Omega))^{-2s}\right)
$$
which is a contradiction since $u(x_0)-v(x_0)>0$ and by using the assumption on 
$\left\|c^+\right\|_{\infty}$.
\end{proof}

We now recall the definition of minimum principle. 
\begin{definition}
We say that the operator $\mathcal{K}$ satisfies the weak minimum principle in $\Omega$ if 
\[ \mathcal{K} u  \le 0 \text{ in } \Omega, \quad u \ge 0 \text{ in } \R^N \setminus \Omega \quad \Longrightarrow \quad u \ge 0 \text{ in } \Omega, \]
and it satisfies the strong minimum principle in $\Omega$ if 
\[ \mathcal{K} u  \le 0 \text{ in } \Omega, \quad u \ge 0 \text{ in } \R^N \quad \Longrightarrow \quad u > 0 \text{ or } u \equiv 0 \text{ in } \Omega. \]
\end{definition}
Similarly, one defines the maximum principle. 
The weak minimum/maximum principle follows by applying the comparison principle Theorem \ref{comparison} with $v=0$ or $u=0$. However, the operators $\K_{k_1, \dots, k_\ell}^\pm$ do not always satisfy the strong maximum or minimum principle, as it is clarified in the next theorem. 

\begin{theorem}\label{SMP}
The following conclusions hold.
\begin{enumerate}
\item[(i)] The operators $\K_{k_1, \dots, k_\ell}^-$, with $k < N$, do not satisfy the strong minimum principle in $\Omega$.
\item[(ii)] The operators $\K_{k_1, \dots, k_\ell}^-$ with $k=N$ satisfy the strong minimum principle in $\Omega$. 
\item[(iii)] The operators $\K_{k_1, \dots, k_\ell} ^+$  satisfy the following implication
\[ \mathcal{K} u(x) \le 0 \text{ in } \Omega, \quad u \ge 0 \text{ in }\R^N \; \Rightarrow \;
u > 0 \text{ in } \Omega \text{ or } u \equiv 0 \text{ in } \R^N. \]
In particular, they satisfy the strong minimum principle. 
\item[(iv)] Let $k<N$, or $k=N$ and $\ell>1$. There exist functions $u$ such that $\K_{k_1, \dots, k_\ell} ^- u \le 0$ in $ \Omega$, $u \equiv 0$ in $\overline \Omega$, and $u \not\equiv 0$ in $\R^N \setminus \overline \Omega$, namely these operators do not satisfy the implication in (iii). 
\end{enumerate}
\end{theorem}

\begin{remark}
The case $k=N$, $\ell=1$, not treated in item \textit{(iv)}, corresponds to the fractional Laplacian, for which the implication in \textit{(iii)} is already known \cite[Corollary 4.2]{MusinaNaz}. 
\end{remark}

\begin{remark}
We notice that since $\K_{k_1, \dots, k_\ell}^+ (-u)= -\K_{k_1, \dots, k_\ell}^- u$, corresponding results hold for the maximum principle. 
\end{remark}

\begin{remark}
We wish to cite here the recent paper \cite{BGI22}, in which
the geometry of the sets of minima for supersolutions of equations involving the operators $\I_k^\pm$ is characterized. 
\end{remark}

\begin{proof}
(i) We slightly adapt the counterexample in Proposition 2.2 in \cite{BGT}. Assume without loss of generality that $0 \in \Omega$. 
Let $\varphi$ be a smooth bounded function of one variable which attains the minimum in $0$. Set $u(x)=\varphi(x_N)$. Then, recalling $k < N$, 
\[ u(x+ \sum_{j=1}^{k_i}  \tau_j e_{a_i(j)})=u(x) \, \text{ for } i=1, \dots, \ell \]
where
\begin{equation}\label{aij} a_i(j):=a(j+\sum_{t=1}^{i-1} k_t). \end{equation}
Hence $\JJ_{V_i}=0$ for any $i=1, \dots, \ell$ with $V_i= \langle e_{a_i(j)} \rangle_{j=1, \dots, k_i}$, thus 
\[ \K_{k_1, \dots, k_\ell}^- u(x) \le 0. \]

(ii)
Let us assume that $u$ satisfies
\[ \begin{cases}
\K_{k_1, \dots, k_\ell}^- u  \le 0 &\text{ in } \Omega \\
u \ge 0 &\text{ in } \R^N
\end{cases} \] 
with $k=N$, 
and let $u(x_0)=0$ for some $x_0 \in \Omega$. We want to prove that $u \equiv 0$ in $\Omega$. Let us proceed by contradiction, and assume there exists $y \in \Omega$ such that $u(y) >0$. Let us choose a ball $B_R(y)$ such that 
\begin{itemize}
\item $B_R(y) \subset \Omega$
\item there exists $x_1 \in \partial B_R(y)$ such that $u(x_1)=0$
\item $u(x) >0$ for all $x \in \overline B_R(y) \setminus \{ x_1 \}$
\end{itemize}
Then, by definition of viscosity super solutions, for fixed $\rho>0$ and $\varphi \in C^2(B_\rho(x_1))$, for which $x_1$ is a minimum point for $u-\varphi$, and for every $\varepsilon>0$, there exists a orthonormal basis $ \{ \xi_1, \dots, \xi_N \}=\{ \xi_1(\varepsilon), \dots, \xi_N(\varepsilon) \}$ such that 
\begin{align}\label{eqn:IN strong} \varepsilon \ge \sum_{i=1}^\ell  &C_{k_i,s}\Big(\int_{B_\rho(0)}\frac{\varphi(x_1+\sum_{j=1}^{k_i} \tau_j \xi_{a_i(j)}) - \varphi(x_1)}{(\sum_{j=1}^{k_i} \tau_j^2)^{\frac{k_i+2s}{2}}} \, d\tau_1 \dots d \tau_{k_i}  \\
\nonumber &+ \int_{B_\rho(0)^c} \frac{u(x_1+\sum_{j=1}^{k_i} \tau_j \xi_{a_i(j)})-u(x_1)}{(\sum_{j=1}^{k_i} \tau_j^2)^{\frac{k_i+2s}{2}}} \, d\tau_1 \dots d \tau_{k_i}\Big),
\end{align}
with $a_i(j)$ as defined in \eqref{aij}. 

Moreover, we know that there exists $\bar j=\bar j(\varepsilon)$ such that 
\begin{equation}\label{direzione entrante} \langle \xi_{\bar j}, \widehat{y-x_1} \rangle \ge \frac{1}{\sqrt{N}}, \quad \text{ with } \widehat{y-x_1} =\frac{y-x_1}{\abs{y-x_1}}. \end{equation}
Notice that there exists $\bar i$ such that $\bar j \in \{a_{\bar i}(1), \dots, a_{\bar i}(k_{\bar i}) \}$. 
Also, we can choose different orthonormal vectors $\{\bar \xi_1, \dots, \bar \xi_{N} \}$ such that $\{ \bar \xi_{a_{\bar i}(1)}, \dots, \bar \xi_{a_{\bar i}(k_{\bar i})} \}$ span the same 
$k_{\bar i}$-dimensional space as $\{ \xi_{a_{\bar i}(1)}, \dots, \xi_{a_{\bar i}(k_{\bar i})} \}$, let us call it $V_{\bar i}$, and such that 
\begin{equation}\label{scelta xi} \langle \bar \xi_{a_{\bar i}(j)}, \widehat{y-x_1} \rangle  \ge \frac{1}{\sqrt{k_{\bar i} N}}, \quad \text{for any $j=1, \dots, k_{\bar i}$}. \end{equation}
Indeed, this is obvious if $k_{\bar i}=1$. If $k_{\bar i}>1$, the intersection of $B_{R}(y)$ with $V_{\bar i}$ is a $k_{\bar i}$ dimensional sphere. Then, one can take the $k_{\bar i}$ dimensional cube inscribed in this sphere having $x_1$ as a vertex, and consider as $\bar \xi_{a_{\bar i}(1)}, \dots, \bar \xi_{a_{\bar i}(k_{\bar i})}$ 
the directions of the edges starting from $x_1$. Notice that in the worst case, the main diagonal of this cube has length $\frac{2R}{\sqrt{N}}$ due to \eqref{direzione entrante}. This gives the estimate \eqref{scelta xi}, see also Figure \ref{max}. 

 \begin{figure}
\centering
\includegraphics[width=0.7\textwidth]{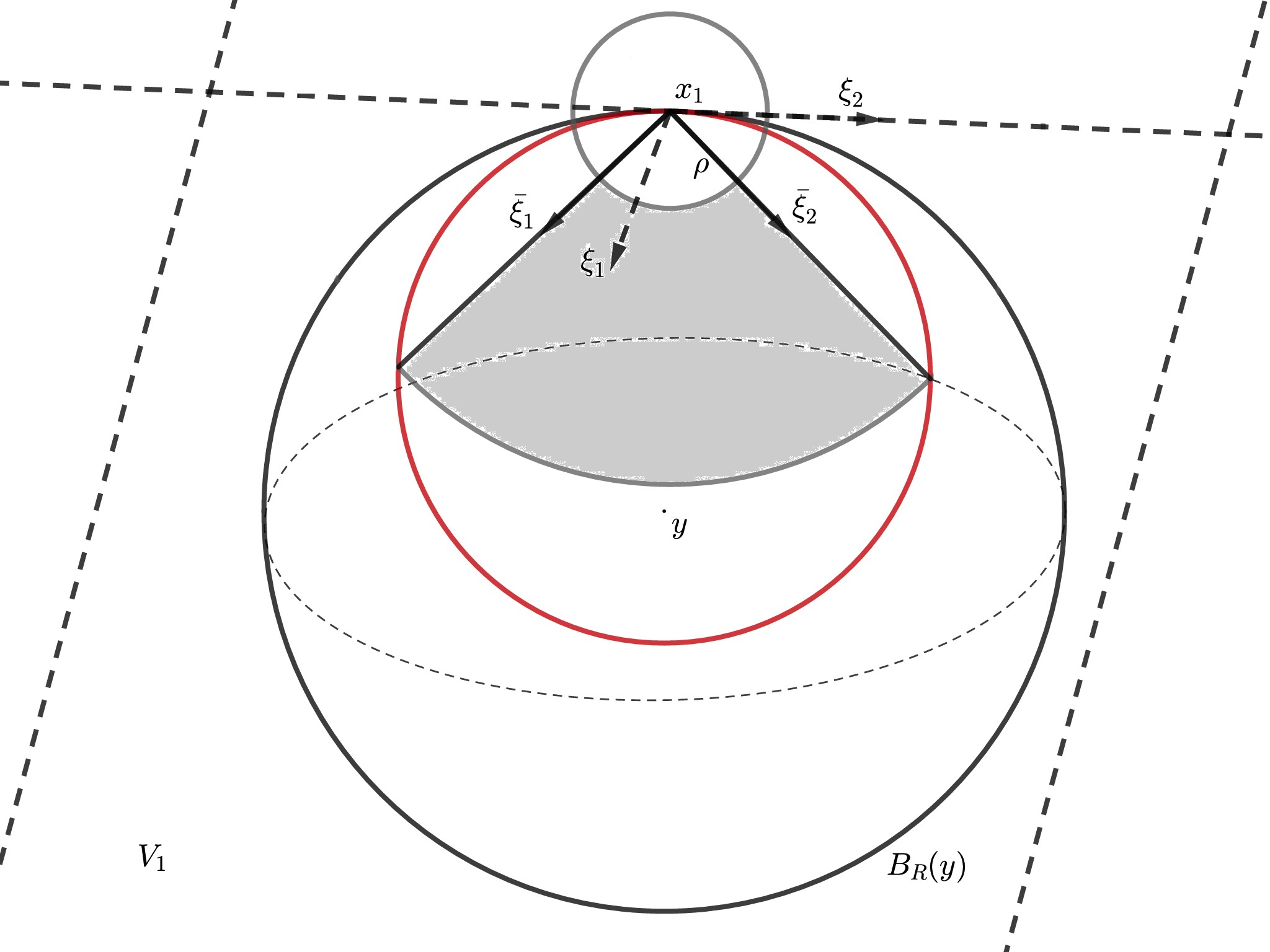}
\caption{Case $N=3$, $\bar i=1$, $k_{\bar i}=2$.  The intersection between the plane $V_1$ and the sphere $B_R(y)$ is represented by the red circle. Moreover, the grey region represents the points $x_1+\sum_{j=1}^{k_{\bar i}} \tau_j \bar \xi_{a_{\bar i}(j)}=x_1 + \tau_1 \bar \xi_1 + \tau_2 \bar \xi_2$ such that $\rho < \abs{\tau} < 2R/ \sqrt{k_{\bar i} N}$, $\tau_i >0$, $i=1, 2$, in the worst case, namely when the diameter of the red circle is $2R/\sqrt N$. }
\label{max}
\end{figure}

Fix $\rho < \frac{2R}{\sqrt{k_{\ell} N}}$, and choose $\varphi \equiv 0$ on $B_\rho(x_1)$. Notice that the choice of $\rho$ does not depend on $\varepsilon$.
In particular, one has $\rho <2R\langle \bar \xi_{a_{\bar i}(j)}, \widehat{y-x_1} \rangle$ for any $j=1, \dots, k_{\bar i}$. 
%%%%
Then, taking into account that $u(x_1)=0$ and $u \ge 0$, from \eqref{eqn:IN strong} one has
\begin{align*} \varepsilon &\ge \sum_{i=1}^\ell  C_{k_i,s}\int_{B_\rho(0)^c} \frac{u(x_1+\sum_{j=1}^{k_i} \tau_j \bar \xi_{a_{i}(j)})}{(\sum_{j=1}^{k_i} \tau_j^2)^{\frac{k_i+2s}{2}}} \, d\tau_1 \dots d \tau_{k_i} \\
& \ge  C_{k_{\bar i},s} \int_{B_\rho(0)^c} \frac{u(x_1+\sum_{j=1}^{k_{\bar i}} \tau_j \bar \xi_{a_{\bar i}(j)})}{(\sum_{j=1}^{k_{\bar i}} \tau_j^2)^{\frac{k_{\bar i}+2s}{2}}} \, d\tau_1 \dots d \tau_{k_{\bar i}} \\
& \ge C_{k_{\bar i},s} \int_{\{\rho < \abs{\tau} < \frac{2R}{\sqrt{k_{\bar i} N}}, \tau_i >0, \forall i\}} \frac{u(x_1+\sum_{j=1}^{k_{\bar i}} \tau_j \bar \xi_{a_{\bar i}(j)})}{(\sum_{j=1}^{k_{\bar i}} \tau_j^2)^{\frac{k_{\bar i}+2s}{2}}} \, d\tau_1 \dots d \tau_{k_{\bar i}} 
\end{align*}

Notice that if 
\[ \tau \in \left\{\rho < \abs{\tau} < \frac{2R}{\sqrt{k_{\bar i} N}}, \tau_i >0, \forall i \right\}, \]
then 
\[ x_1+\sum_{j=1}^{k_{\bar i}} \tau_j \bar \xi_{a_{\bar i}(j)} \in \overline B_R(y) \setminus B_\rho(x_1)\]
 due to \eqref{scelta xi}, see also Figure \ref{max}. Thus, 

\begin{align*}
\varepsilon &\ge C_{k_{\bar i},s}  \left( \min_{\overline B_R(y) \setminus B_\rho(x_1) } u  \right)\; \int_{\{\rho < \abs{\tau} < \frac{2R}{\sqrt{k_{\bar i} N}}, \tau_i >0, \forall i\}}   \frac{1}{\abs{\tau}^{k_{\bar i}+2s}} \, d\tau_1 \dots d \tau_{k_{\bar i}}\\
 &= \frac1{2^{k_{\bar i}}} C_{k_{\bar i},s}  \left( \min_{\overline B_R(y) \setminus B_\rho(x_1) } u \right) \int_{\rho < \abs{\tau} < \frac{2R}{\sqrt{k_{\bar i} N}}}   \frac{1}{\abs{\tau}^{k_{\bar i}+2s}} \, d\tau_1 \dots d \tau_{k_{\bar i}}\\
 & \ge \min_i \frac{C_{k_{i},s}}{2^{k_i}} \left( \min_{\overline B_R(y) \setminus B_\rho(x_1) } u\right) \int_{\rho}^{\frac{2R}{\sqrt{k_{\ell} N}}} \frac{1}{r^{1+2s}} \, dr\\
 &=C(\rho, R, N, s) \min_{\overline B_R(y) \setminus B_\rho(x_1) } u >0.
\end{align*}

This gives the contradiction if $\varepsilon$ is small enough.

\textit{(iii)} Take $u$ which satisfies the assumptions of the minimum principle, and assume there exists $x_0 \in \Omega$ such that $u(x_0)=0$. Choose any orthonormal basis of $\R^N$ $\{ \xi_1, \dots, \xi_{N} \}$. Thus, 
\[ 0 \ge \K_{k_1, \dots, k_\ell}^+ u(x_0)  \ge \sum_{i=1}^\ell \mathcal{J}_{V_i} u(x_0) 
 = \sum_{i=1}^\ell C_{k_i, s} \int_{\R^{k_i}} \frac{u(x_0 + \sum_{j=1}^{k_i} \tau_j \xi_j^i ) }{|\tau|^{k_i+2s}} \, d\tau.
\]
Hence, since $u \ge 0$ in $\R^N$, we conclude that $u \equiv 0$ on every space $V_i  + x_0$. Since the directions $\xi_i$ are arbitrary, we get $u \equiv 0$ on $\R^N$. 

\textit{(iv)} Notice that if $k < N$ this is immediate as the minimum principle is not satisfied. Let $k=N$ and $\ell >1$, and take 
\[ u(x)=\begin{cases}
0  &\text{ if } \exists i=1, \dots, \ell \text{ such that } x \in \langle e_j \rangle_{j \in \mathcal{A}_i}, \text { or if } x \in \overline B_1(0)\\
1 &\text{ otherwise,} 
\end{cases} \]
where $\mathcal{A}_i$ where defined in \eqref{Ai}, 
and notice that $\JJ_{\langle e_j \rangle_{j \in \mathcal{A}_i}}u(x)=0$, thus
\[ \K_{k_1, \dots, k_\ell}^- u (x)\le 0 \text{ in } B_1(0). \]
Moreover, $u \equiv 0$ in $\overline B_1(0)$, however $u \not\equiv 0$ in $\R^N \setminus \overline B_1(0)$ as $\ell >1$. 
\end{proof}

\section{Hopf type Lemma}\label{sec:hopf}
The next lemma provides a suitable barrier function in the spirit of Section 3.6 in \cite{BV}.  
\begin{lemma}\label{funzione barriera}
If $V$ is a $k$-dimensional space of $\R^N$, then 
\[ \JJ_V (R^2-\abs{x}^2)^s_+= - C_{k,s} \beta(1-s, s) \frac{\omega_k}2 \, \text{ in } B_R(0), \]
where 
\[ \beta(1-s, s)=\int_0^1 t^{-s} (1-t)^{s-1} \, dt \]
is the Beta function. 
In particular, 
\[ \K_{k_1, \dots, k_\ell}^+ (R^2-\abs{x}^2)^s_+= \K_{k_1, \dots, k_\ell}^- (R^2-\abs{x}^2)^s_+= - \frac{\beta(1-s, s)}2  \sum_{i=1}^\ell \, C_{k_i, s} \omega_{k_i} \, \text{ in } B_R(0). \]
\end{lemma}

\begin{proof}
Let $R=1$ for simplicity. Let $\{ \xi_1, \dots, \xi_k \}$ a $k$-dimensional orthonormal basis, and 
\[ t_i=-\langle x, \xi_i \rangle + \tau_i \sqrt{\sum_{j=1}^k \langle x, \xi_j \rangle^2 - \abs{x}^2+1} \]
thus
\[ (1-\sum_{i=1}^k \tau_i^2)(1-\abs{x}^2+ \sum_{i=1}^k \langle x, \xi_i \rangle^2) = 1-\abs{x}^2-\sum_{i=1}^k t_i^2 - 2 \sum_{i=1}^k t_i \langle x, \xi_i \rangle \]
Set 
\[ \tilde u : \R^k \to \R \quad \tilde u (t)= (1- \abs{t}^2)_+^s \]
and  
\[y_i=\frac{\langle x, \xi_i \rangle}{\sqrt{1-\abs{x}^2 + \sum_{j=1}^k \langle x, \xi_j \rangle^2}}.  \] 
One has for $x \in B_1(0)$  
\[ \begin{split}
\JJ_Vu(x)&=\int_{\R^k} \frac{(1-\abs{x+\sum_{i=1}^k t_i \xi_i}^2)_+^s- (1-\abs{x}^2)^s}{(\sum_{i=1}^k t_i^2)^{\frac{2s+k}{2}}} \, d t_1 \dots d t_k \\
&= \int_{\R^k} \frac{  (1-\sum_{i=1}^k \tau_i^2)_+^s-  (1- \sum_{i=1}^k y_i^2)^s}{(\sum_{i=1}^k (\tau_i - y_i)^2)^{\frac{2s+k}{2}}} \, d\tau_1 \dots d\tau_k \\
&= -(-\Delta)^s \tilde u(y_1, \dots, y_k)=- C_{k, s} \beta(1-s, s) \frac{\omega_k}2. \qedhere
\end{split}  \]
\end{proof}

 The proof of the following Proposition is inspired by \cite[Proposition 4.9]{BGS}, see also \cite{GrecoServadei}. 

\begin{proposition}\label{hopf}
Let $\Omega$ be a bounded $C^2$ domain, $k=N$, and let $u$ satisfy  
\[ \begin{cases}
\K_{k_1, \dots, k_\ell}^- u \le 0 &\text{ in } \Omega \\
u \ge 0 &\text{ in } \R^N \setminus \Omega.
\end{cases}\]
Assume $u \not \equiv 0$ in $\Omega$. Then there exists a positive constant $c=c(\Omega, u)$ such that 
\begin{equation}\label{eq hopf} u(x) \ge c\, d(x)^s\quad\forall x\in\overline\Omega. \end{equation} 
\end{proposition}

Notice that the conclusion is not true for the operators $\K_{k_1, \dots, k_\ell}^-$, $k < N$. 
Indeed, consider the function
\[ u(x)=\begin{cases}
e^{-\frac{1}{1-\abs{x}^2}} &\text{ if } \abs{x} < 1 \\
0 &\text{ if } \abs{x} \ge 1 \end{cases} \]
fix $x \in B_1(0)$, and take $\{ \xi_i \} \in \mathcal{V}_k$ such that $\langle x, \xi_i \rangle=0$ for any $i=1, \dots, k$. Hence
\[ \abs{x+\sum_{j=1}^{k_i} \tau_j \xi_{a_i(j)}}^2 =\abs{x}^2 + |\tau|^2 \ge \abs{x}^2 \]
where $a_i(j)$ is defined in \eqref{aij}, and using  the radial monotonicity of $u$ 
\[ \K_{k_1, \dots, k_\ell}^- u(x) \le  0 \text{ in } B_1(0). \]
However, $u$ clearly does not satisfy 
\[ u(x) \ge c\, d(x)^\gamma \]
for any positive constants $c, \gamma$. 

As a consequence of Proposition \ref{hopf}, we immediately obtain the following 
\begin{corollary}\label{cor hopf}
Let $\Omega$ be a bounded $C^2$ domain, and let $u$ satisfy  
\[ \begin{cases}
\K_{k_1, \dots, k_\ell}^+ u \le 0 &\text{ in } \Omega \\
u \ge 0 &\text{ in } \R^N \setminus \Omega.
\end{cases}\]
Assume $u \not \equiv 0$ in $\Omega$. Then  
\[ u(x) \ge c\, d(x)^s \]
for some positive constant $c=c(\Omega, u)$. 
\end{corollary}

\begin{proof}[Proof of Proposition \ref{hopf}]
By the weak and strong minimum principles, see Theorem \ref{comparison} and Theorem \ref{SMP}-(ii), $u>0$ in $\Omega$. Therefore, for any $K$ compact subset of $\Omega$ we have
\begin{equation}\label{bound} \inf_{y \in K} u(y) >0. \end{equation} 
 Without loss of generality we can further assume that $u$ vanishes somewhere in $\partial\Omega$, otherwise the conclusion is obvious.\\
Since $\Omega$ is a $C^2$ domain, there exists a positive constant $\varepsilon$, depending on $\Omega$, such that for any $x \in \Omega_\varepsilon=\{ x \in \Omega: d(x) < \varepsilon \}$ there are a unique $z \in \partial \Omega$ for which $d(x)=\abs{x-z}$ and a ball $B_{2\varepsilon}(\bar y) \subset \Omega$ such that $\overline{B_{2\varepsilon}(\bar y)} \cap (\R^N \setminus \Omega)=\{ z \}$.  \\
Now we consider the radial function $w(x)={((2\varepsilon)^2-\abs{x- \bar y}^2)}^s_+$ which satisfies, see  Lemma \ref{funzione barriera}, the equation 
\[ \K_{k_1, \dots, k_\ell}^- w=- \, \sum_{i=1}^{\ell} C_{k_i, s} \frac{\omega_{k_i}}{2} \beta(1-s, s) \, \text{ in } B_{2\varepsilon}(\bar y). \]
We claim that there exists $\bar n= \bar n(u, \varepsilon)$ such that 
\[ u \ge w_{\bar n} \text{ in } \R^N, \]
where 
\[ w_n(x) = \frac 1n w(x). \]
This implies \eqref{eq hopf}. Indeed, for any $x\in\Omega_\varepsilon$
\begin{equation}\label{2406eq1} w_{\bar n} (x)=\frac 1{\bar n}  ((2\varepsilon)^2-\abs{x- \bar y}^2)^s_+ \ge \frac {2\varepsilon}{\bar n} \abs{x-z}^s =\frac {2\varepsilon}{\bar n} d(x)^s, \end{equation}
and
 \begin{equation}\label{2406eq2}
u(x)\geq \min_{y\in\Omega\backslash\Omega_\varepsilon}\frac{u(y)}{d(y)^s}d(x)^s\quad\forall x\in \Omega\backslash\Omega_\varepsilon.
\end{equation}
 From \eqref{2406eq1}-\eqref{2406eq2} we obtain \eqref{eq hopf} with $c=\min\left\{\frac {2\varepsilon}{\bar n} ,\min_{y\in\Omega\backslash\Omega_\varepsilon}\frac{u(y)}{d(y)^s}\right\}$.

We proceed by contradiction in order to prove the claim, hence, we suppose that for any $n \in \mathbb{N}$ 
\[ v_n = w_n-u \]
is USC and positive somewhere. From now on, for simplicity of notation, we assume that $B_{2\varepsilon}(\bar y) =B_1(0)$. 
Since
\[ w_n =0 \le u \text{ in } \R^N \setminus B_1(0), \]
we know that it attains its positive maximum  $x_n$ in $B_1(0) \subset \Omega$. One has 
\[ 0 <u(x_n) < w_n (x_n). \]
Also, $w_n \to 0$ uniformly in $\R^N$, thus
\begin{equation}\label{2406eq3} \lim_{n \to +\infty} u(x_n) =0. \end{equation}
Therefore, recalling \eqref{bound}, $ \abs{x_n} \to 1 $
as $n \to \infty$, hence in particular $x_n \in B_1(0) \setminus B_{r_0}(0)$, where $r_0=\sqrt{1 - \frac 1{2N}}$, and $d(x_n) < (1-r_0)/2$ for $n$ large enough.

Since $\K_{k_1, \dots, k_\ell}^- u \le 0$ in $\Omega$,  we know that for any $n \in \N$ there exists $\{ \xi_1(n), \dots, \xi_N(n) \}$ orthonormal basis of $\R^N$ such that 
\begin{align}\label{u supersol hopf} \sum_{i=1}^\ell  &C_{k_i,s}\Big(\int_{B_\rho(0)}\frac{\varphi(x_n+\sum_{j=1}^{k_i} \tau_j \xi_{a_i(j)} (n)) - \varphi(x_n)}{(\sum_{j=1}^{k_i} \tau_j^2)^{\frac{k_i+2s}{2}}} \, d\tau_1 \dots d \tau_{k_i}  \\
\nonumber &+ \int_{B_\rho(0)^c} \frac{u(x_n+\sum_{j=1}^{k_i} \tau_j \xi_{a_i(j)}(n))-u(x_n)}{(\sum_{j=1}^{k_i} \tau_j^2)^{\frac{k_i+2s}{2}}} \, d\tau_1 \dots d \tau_{k_i}\Big) \le \frac 1n, \end{align}
where $a_i(j)$ is defined in \eqref{aij}. 
Since $\{ \xi_1(n), \dots, \xi_N(n) \}$ is a basis of $\R^N$, then there exists at least one $\xi_{\bar j}(n)$ such that  $\langle \hat x_n, \xi_{\bar j}(n) \rangle \ge \frac{1}{\sqrt N}$. We can assume without loss of generality that this $\xi_{\bar j}$ appears in $\JJ_{V_1}$, namely $\bar j \le k_1$, and also we can suppose that for each $j \le k_1$ one has $\langle \hat x_n, \xi_j (n) \rangle \ge \frac{1}{\sqrt{k_1 N}}$ (this is true up to choosing a different orthonormal basis for $V_1$, see the discussion after \eqref{scelta xi}).

Let us choose $\rho = d(x_n) < (1- r_0)/2$, and $\varphi(x)=w_n (x) \in C^2(B_\rho(x_n))$ as test function. 
We consider the left hand side of \eqref{u supersol hopf}, and we aim at providing a positive lower bound independent on $n$, which will give the desired contradiction. 

Let us start with the second integral in \eqref{u supersol hopf} for each fixed $i=2, \dots, \ell$, and let us notice that since $x_n$ is a maximum point for $v_n$ 
\[ \int_{B_\rho(0)^c} \frac{u(x_n+\sum_{j=1}^{k_i} \tau_j \xi_{a_i(j)}(n))-u(x_n)}{(\sum_{j=1}^{k_i} \tau_j^2)^{\frac{k_i+2s}{2}}}  \ge \int_{B_\rho(0)^c} \frac{w_n(x_n+\sum_{j=1}^{k_i} \tau_j \xi_{a_i(j)}(n))-w_n(x_n)}{(\sum_{j=1}^{k_i} \tau_j^2)^{\frac{k_i+2s}{2}}} . \]
On the other hand, in order to estimate the integral for $i=1$, we first  define
\[ A_1(n)= \{ \tau \in B_\rho(0)^c: t_1(n) < \tau_i < t_2(n), i=1, \dots, k_1 \}. \]
and
\[ A_2(n) = A_1(n)^c \cap B_\rho(0)^c, \]
where
\[ \sqrt{k_1} t_1(n) =-\frac{|x_n|}{\sqrt N} - \sqrt{1- \frac1{2N} - |x_n|^2\left(1-\frac 1 N \right)}, \]
\[ \sqrt{k_1} t_2(n) =-\frac{|x_n|}{\sqrt N} +  \sqrt{1- \frac1{2N} - |x_n|^2\left(1-\frac 1 N \right)}. \]

Then we make the following decomposition 
\begin{equation}\label{int2} \int_{B_\rho(0)^c} \frac{u(x_n+\sum_{j=1}^{k_1} \tau_j \xi_{j}(n))-u(x_n)}{(\sum_{j=1}^{k_1} \tau_j^2)^{\frac{k_1+2s}{2}}} =J_1(n)+ J_2(n), \end{equation}
where  
\[ J_1(n)= \int_{A_1(n)} \frac{u(x_n+\sum_{j=1}^{k_1} \tau_j \xi_{j}(n))-u(x_n)}{(\sum_{j=1}^{k_1} \tau_j^2)^{\frac{k_1+2s}{2}}}, \]
\[ J_2(n)= \int_{A_2(n)} \frac{u(x_n+\sum_{j=1}^{k_1} \tau_j \xi_{j}(n))-u(x_n)}{(\sum_{j=1}^{k_1} \tau_j^2)^{\frac{k_1+2s}{2}}}.\]

Integral $J_2$ can be estimated as above, using the fact that $x_n$ is a maximum point for $v_n$.

In order to estimate $J_1$, we preliminary notice that if $\tau \in A_1(n)$, then $x_n+\sum_{j=1}^{k_1} \tau_j \xi_{j}(n) \in B_{r_0}(0)$. Indeed, for each $j$ fixed
\[ \tau_j^2 + \frac{2 \tau_j |x_n|}{\sqrt{k_1 N}} \le \frac{1}{k_1} \left(1- \frac 1{2N} - |x_n|^2\right), \]
thus, as $\tau_j < t_2(n) <0$, 
\[ \abs{x_n+\sum_{j=1}^{k_1} \tau_j \xi_{j}(n)}^2 \le  \abs{x_n}^2+ |\tau|^2 + \frac{2 \sum_{j=1}^{k_1} \tau_{j} \abs{x_n}}{\sqrt{k_1 N}} \le 1 -\frac{1}{2N} = r_0^2, \] 
Also, for $n$ large we can assume $\rho=d(x_n) $ small enough, and such that 
\[  \{ \tau: t_1(n) < \tau_i < t_2(n), i=1, \dots, k_1 \} \subset B_\rho(0)^c, \]
since as $n\to+\infty$, $d(x_n)\to0$, and 
\[ t_1(n) \to t_1^0= \frac{1}{\sqrt{k_1 N}} \left( -1- \frac{1}{\sqrt 2} \right), \]
\[  t_2(n) \to t_2^0= \frac{1}{\sqrt{k_1 N}} \left( -1+\frac{1}{\sqrt 2} \right). \]

We now use the fact that 
$u(x_n +\sum_{j=1}^{k_1} \tau_j \xi_{j}(n))\geq\min_{\overline B_{r_0}} u >0$. 
We obtain
\begin{align*} J_1(n) & \ge  \int_{A_1(n)}  \frac{u(x_n +\sum_{j=1}^{k_1} \tau_j \xi_{j}(n)) - 2u(x_n)}{(\sum_{j=1}^{k_1} \tau_j^2)^{\frac{k_1+2s}{2}}}\,d\tau \\
 &\ge \left( \min_{\overline B_{r_0}} u-2u(x_n)\right)\,  \int_{\{ t_1(n) < \tau_j < t_2(n), \, \forall j\}}  \frac{1}{(\sum_{j=1}^{k_1} \tau_j^2)^{\frac{k_1+2s}{2}}}. 
\end{align*}

Now, putting estimates above together and recalling \eqref{u supersol hopf}, one has
\begin{equation}\label{eq:hopf} \begin{split} \frac 1n \ge & \sum_{i=1}^\ell \JJ_{V_i} w_n  - \int_{A_1(n)}  \frac{w_n(x_n+\sum_{j=1}^{k_1} \tau_j \xi_{j}(n))-w_n(x_n)}{(\sum_{j=1}^{k_1} \tau_j^2)^{\frac{k_1+2s}{2}}}  \,d\tau_1 \dots d\tau_{k_1}  \\
&\hspace{2cm}+ \left(\min_{\overline B_{r_0}} u-2u(x_n)\right) \int_{\{ t_1(n) < \tau_j < t_2(n), \, \forall j\}}  \frac{1}{(\sum_{j=1}^{k_1} \tau_j^2)^{\frac{k_1+2s}{2}}} .\end{split}\end{equation}

Observe that by the dominated convergence theorem 
\[ C(n):= \int_{\{ t_1(n) < \tau_j < t_2(n), \, \forall j\}}  \frac{1}{(\sum_{j=1}^{k_1} \tau_j^2)^{\frac{k_1+2s}{2}}} \to \int_{\{ t_1^0 < \tau_j < t_2^0, \, \forall j\}}  \frac{1}{(\sum_{j=1}^{k_1} \tau_j^2)^{\frac{k_1+2s}{2}}}  >0. \]

Notice that, as $n\to+\infty$
\[ \abs{ \int_{A_1(n)}  \frac{w_n(x_n+\sum_{j=1}^{k_1} \tau_j \xi_{j}(n))-w_n(x_n)}{(\sum_{j=1}^{k_1} \tau_j^2)^{\frac{k_1+2s}{2}}}  \,d\tau_1 \dots d\tau_{k_1} } 
\le \frac{2}{n} C(n) \to 0, \]
and that by Lemma \ref{funzione barriera}
\[ \sum_{i=1}^\ell \JJ_{V_i} w_n  = - \frac{1}{n} \sum_{i=1}^\ell \frac{\omega_{k_i}}{2} C_{k_i, s} \beta(1-s, s) \to 0. \]
Thus by taking the limit $n \to +\infty$ in \eqref{eq:hopf} and using \eqref{2406eq3} we get the contradiction
\[ 0 \ge  \min_{\overline B_{r_0}} u \int_{\{ t_1^0 < \tau_j < t_2^0, \, \forall j\}}  \frac{1}{(\sum_{j=1}^{k_1} \tau_j^2)^{\frac{k_1+2s}{2}}}   > 0 \,. \qedhere\] 
\end{proof}

\section{Existence for the Dirichlet problem and principal eigenvalues }\label{sec:dirichlet}
We now recall some stability results, which have been proved in a very general context in \cite{BCI, BarlesImbert}, see also \cite{AlvarezTourin}. In \cite{BGS} we give a simplified proof for $\I_k^\pm$, which works with obvious adaptation in the setting of the present paper as well.  

For the local counterparts, we refer to \cite{CIL}. 
Let us set 
\[ u_* (x)= \sup_{r>0} \inf_{\abs{y-x} \le r} u(y), \quad u^*(x) = \inf_{r>0} \sup_{\abs{y-x} \le r} u(y) \]
and 
\[ {\liminf}_* u_n(x)=\lim_{j \to \infty} \inf \left\{u_n(y): n \ge j, \, \abs{y-x} \le \frac1j \right\} , \]
\[ {\limsup}^* u_n (x)= \lim_{j \to \infty} \sup \left\{u_n(y): n \ge j, \, \abs{y-x} \le \frac1j \right\} .\]

\begin{lemma}\label{stability1}
Let $u_n \in USC(\Omega)$ (respectively $LSC(\Omega)$) be a sequence of subsolutions (supersolutions) of
\begin{equation}\label{limit j} \K_{k_1, \dots, k_\ell}^\pm u_n = f_n(x)  \text{ in } \Omega,  \end{equation}
where $f_n$ are locally uniformly bounded functions, 
and $u_n \le 0$ ($u_n \ge 0$) in $\R^N \setminus \Omega$. We assume that there exists $M>0$ such that  for any $n \in \N$ 
\begin{equation}\label{bound uj} \norm{u_n}_\infty \le M \text{ in } \R^N. \end{equation}
Then $\overline u:= {\limsup}^* u_n$ (resp. $\underline u :={\liminf}_* u_n$) is a subsolution (resp. supersolution) of 
\[ \K_{k_1, \dots, k_\ell}^\pm u = f(x) \text{ in } \Omega,  \]
such that $u \le 0$ ($u \ge 0$)  in $\R^N \setminus \overline \Omega$, where $f=\liminf_* f_n$ (resp. $f=\limsup^* f_n$).
\end{lemma}
\begin{lemma}\label{stability2}
Let $(u_\alpha)_\alpha \subseteq USC(\Omega)$ (respectively $LSC(\Omega)$) a family of subsolutions (supersolutions) of 
\[ \K_{k_1, \dots, k_\ell}^\pm u_\alpha = f_\alpha(x)\text{ in } \Omega  \]
such that 
$u_\alpha \le 0$ ($u_\alpha \ge 0$)  in $\R^N \setminus \Omega$, and  there exists $M>0$ such that for any $\alpha $
\[ \norm{u_\alpha}_\infty \le M \text{ in } \R^N,  \]
where $f_\alpha$ are uniformly bounded. 
Set $u=\sup_\alpha u_\alpha$ (resp. $v=\inf_{\alpha} u_\alpha$). Then $u^*$ (resp $v_*$) is a subsolution (resp supersolution) of 
\[ \K_{k_1, \dots, k_\ell}^\pm u = f(x) \text{ in } \Omega  \]
such that $u \le 0$ ($u \ge 0$)  in $\R^N \setminus \Omega$, 
where $f=(\inf_\alpha f_\alpha)_*$ (resp. $f=(\sup_\alpha f_\alpha)^*$).
\end{lemma}

The following can be seen as an analog of the Perron method.
\begin{lemma}\label{perron}
Let $\underline u$ and $\overline u$ in $C(\R^N)$ be respectively sub and supersolutions of 
\begin{equation}\label{eq perron}
\K_{k_1, \dots, k_\ell}^\pm u =f(x) \text{ in } \Omega,
\end{equation}
such that $\underline u= \overline u=0$ in $\R^N \setminus \Omega$. Then there exists a solution $v \in C(\R^N)$ to \eqref{eq perron} such that $\underline u  \le v \le \overline u$, and $v=0$ in $\R^N \setminus \Omega$. 
\end{lemma}

Using stability properties above, we prove existence of a unique solution to the Dirichlet problem in uniformly convex domains, namely domains of the following form 
$$
\Omega=\bigcap_{y\in Y}B_R(y).
$$

\begin{theorem}\label{lem dirichlet}
Let $f$ be a bounded continuous function, and let $\Omega$ be a uniformly convex domain. Then there exists a unique function $u \in C(\R^N)$ such that 
\begin{equation}\label{dirichlet} \begin{cases}
\K_{k_1, \dots, k_\ell}^\pm u = f(x) &\text{ in } \Omega \\
u=0 &\text{ in } \R^N \setminus \Omega.
\end{cases} \end{equation}
\end{theorem}
\begin{proof}
Exploiting the barrier functions in Lemma \ref{funzione barriera}, we build suitable sub/super solutions. Indeed, for any $y \in Y$ one considers the function 
\[ v_y(x)=M(R^2-\abs{x-y}^2)^s_+ \]
which for $M=M(k_1, \dots, k_\ell, s)$ big enough satisfies
\[ \K_{k_1, \dots, k_\ell}^+ v_y \le - \norm{f}_\infty \text{ in } B_R(y). \]
We now take 
\begin{equation}\label{v inf} v(x)= \inf_{y \in Y} v_y(x) \end{equation}
which is a supersolution to \eqref{dirichlet}. Analogously we take the supremum of the sub solutions 
\[ w_y(x)=- v_y(x). \]
Notice that 
\[\K_{k_1, \dots, k_\ell}^+ w_y(x) \ge \K_{k_1, \dots, k_\ell}^- w_y(x)=-\K_{k_1, \dots, k_\ell}^+ v_y(x) \ge \norm{f}_\infty \text{ in } B_R(y) \]
for a sufficiently big constant $M$. 

We now exploit the Perron method, applying Lemma \ref{perron}, to get a solution to \eqref{dirichlet}. 
Uniqueness follows from Theorem \ref{comparison}. 
\end{proof}

We finally define the following generalized principal eigenvalues, adapting the classical definition in \cite{BNV}, 
\begin{multline*} \mu_{k_1, \dots, k_\ell}^\pm = \sup \Big \{ \mu :\, \exists v \in LSC(\Omega)\cap L^\infty(\R^N), v>0 \text{ in } \Omega, v \ge 0 \text{ in } \R^N, \\ ¨\K_{k_1, \dots, k_\ell}^\pm v + \mu v \le 0 \text{ in } \Omega \Big \}. \end{multline*}
Also let us set 
\begin{multline*}  \bar\mu^\pm_{k_1, \dots, k_\ell}=\sup\Big\{\mu:\,\exists v\in LSC(\Omega)\cap L^\infty(\R^N),\,\inf_\Omega v>0,\,v\geq0\;\text{in $\R^N$},\;\\ \K_{k_1, \dots, k_\ell}^\pm v+\mu v\leq0  \text{ in } \Omega\Big \}. \end{multline*}

\begin{theorem}\label{max principle bar}
The operators $\K_{k_1, \dots, k_\ell}^\pm(\cdot)+\mu\cdot$ satisfy the maximum principle for $\mu<\bar\mu^\pm_{k_1, \dots, k_\ell}$.
\end{theorem}
The proof is analogous to \cite[Theorem 6.2]{BGS}, up to obvious modifications. 
We finally give some estimates on these generalized eigenvalues. 

\begin{proposition}\label{mu-inf}
One has
\begin{enumerate}
\item[(i)] $\bar\mu_{k_1, \dots, k_\ell}^-=\mu_{k_1, \dots, k_\ell}^-=+\infty$ for any $k < N$. 
\item[(ii)] If $B_{R_1} \subseteq \Omega$, then 
\[ \bar \mu_{k_1, \dots, k_\ell}^- \le \frac{c_1}{R_1^{2s}} < +\infty \]
if $k=N$, and 
\[ \bar \mu_{k_1, \dots, k_\ell}^+ \le \frac{c_1}{R_1^{2s}} < +\infty \]
for any $k_1, \dots, k_\ell$, 
where $c_1>0$. 
\item[(iii)] If $\Omega \subseteq B_{R_2}$, then 
\[0<\frac{c_2}{R_2^{2s}} \sum_{i=1}^\ell k_i \omega_{k_i}  \le \bar\mu_1^+ \sum_{i=1}^\ell k_i \omega_{k_i} \le \bar \mu_{k_1, \dots, k_\ell}^+\le \bar \mu_{k_1, \dots, k_\ell}^-\]
where $c_2>0$ and $\omega_{k_i}$ is the volume of the $k_i$ dimensional sphere. 
\end{enumerate}
\end{proposition}

\begin{proof}
(i) Let $w(x)=e^{-\alpha \abs{x}^2} > 0$ for $\alpha>0$ and fix any $\mu >0$. Notice that 
\[\int_{\R^{k_i}} \frac{1-e^{-\alpha \sum_{j=1}^\ell \tau_j^2}}{(\sum \tau_j^2)^{\frac{k_i+2s}{2}}} \, d\tau_1 \dots d \tau_{k_i} = \alpha^s \int_{\R^{k_i}} \frac{1-e^{-\sum_{j=1}^\ell \tau_j^2}}{(\sum_{j=1}^\ell \tau_j^2)^{\frac{k_i+2s}{2}}} \, d\tau_1 \dots d \tau_{k_i}. \]
Hence, we obtain, choosing $\{ \xi_j \} \in \mathcal{V}_k$ such that $\xi_j$ is orthogonal to $x$ for any $j$, 
\[ \K_{k_1, \dots, k_\ell}^- w(x) + \mu w (x) \le  - \sum_{i=1}^\ell C_{k_i, s} e^{-\alpha \abs{x}^2} \int_{\R^{k_i}} \frac{1-e^{-\alpha \sum_{j=1}^\ell \tau_j^2}}{(\sum_{j=1}^\ell \tau_j^2)^{\frac{k_i+2s}{2}}} \, d\tau_1 \dots d \tau_{k_i} + \mu e^{-\alpha \abs{x}^2} \le 0
\]
if $\alpha$  is big enough. 

(ii) 
Let $k=N$. By scaling we obtain
\[ \bar\mu_{k_1, \dots, k_\ell}^- (\Omega) \le \bar\mu_{k_1, \dots, k_\ell}^- (B_{R_1}) =  \frac{\bar\mu_{k_1, \dots, k_\ell}^- (B_1)}{R_1^{2s}}. \]
Hence it is sufficient to prove that $\bar\mu_{k_1, \dots, k_\ell}^- (B_1)$ is bounded from above. Arguing as in \cite{QSX}, choose a constant function $h \ge 0$, $h \not \equiv 0$ with compact support in $B_1$. By Theorem \ref{lem dirichlet}, there exists a unique solution to the following
\[ \begin{cases}
- \K_{k_1, \dots, k_\ell}^- v = h &\text{ in }B_1 \\
v=0 &\text{ in } \R^N \setminus B_1. 
\end{cases} \]
By Theorem \ref{comparison} and Theorem  \ref{SMP} $v >0$ in $B_1$ (notice that it is crucial here to have $k=N$). Since $h$ has compact support we may select a constant $\rho_0 >0$ such that $\rho_0 v \ge h$ in $B_1$. Therefore, $v$ satisfies
\[ \begin{cases}
 \K_{k_1, \dots, k_\ell}^- v +\rho_0v\geq0 &\text{ in }B_1 \\
v=0 &\text{ in } \R^N \setminus B_1. 
\end{cases} \]
By Theorem \ref{max principle bar} we infer that $\bar \mu_{k_1, \dots, k_\ell}^- \le \rho_0$. 
The same proof shows that $\bar \mu_{k_1, \dots, k_\ell}^+$ is bounded from above for any $k_1, \dots, k_\ell$. 

(iii)
We first note that in the definitions of $\bar\mu^\pm_{k_1, \dots, k_\ell}$ it is not restrictive to suppose $\mu\geq0$ (since the constant function $v \equiv 1$ is a positive solution of $\K_{k_1, \dots, k_\ell}^\pm v=0$). Moreover if $\mu\geq0$ and $v$ is a nonnegative supersolution of the equation $$\I^+_1 v+\mu v=0 \quad\text{in $\Omega$},$$
then $\I^+_1 v\leq0$ in $\Omega$ and we claim 
\begin{equation}\label{bound3}
\JJ_{V_{k_i}} v+k_i \omega_{k_i} \mu v\leq0 \quad\text{in $\Omega$}
\end{equation}
for any $V_{k_i}$ of dimension $k_i$. 
Indeed, if $k_i=1$, it is immediate that
\[ \JJ_{V_1}v(x)=2 \I_{\xi_1}v(x) \le 2 \I_1^+v(x) \le -2 \mu v(x), \]
and \eqref{bound3} is satisfied.
If $k_i \ge 2$, passing to spherical coordinates 
\[ y_1= \tau \cos(\theta_1), y_2=\tau \sin(\theta_1)\cos(\theta_2), \dots, y_{k_i}=\tau \sin(\theta_1)\cdots \sin(\theta_{k_i-2})\sin(\theta_{k_i-1}) \]
where $\tau >0$, $\theta_j \in [0, \pi]$ for any $j=1, \dots, k_i-2$, and $\theta_{k_i-1} \in [0, 2\pi)$, one has 
\[ y= \tau \xi_{\theta_1, \dots, \theta_{k_i-1}}, \text{ where } \xi_{\theta_1, \dots, \theta_{k_i-1}}\in \mathcal{S}^{k_i-1} \subset \R^{k_i},  \]
and
\begin{align*}
\JJ_{V_{k_i}} v (x)&= \int_{\R^{k_i}} \frac{u(x+y) - u(x)}{\abs{y}^{2s+k_i}} \, d \mathcal{H}^{k_i}(y) \\
&= \int_{\R^+ \times [0, 2\pi) \times [0, \pi]^{k_i-2} }  \frac{u(x+\tau \xi_{\theta_1, \dots, \theta_{k_i-1}}) - u(x)}{\tau^{2s+1}} f(\theta_1, \dots, \theta_{k_i}) \, d \tau d \theta_{k_i-1} \dots d \theta_1 \\
&\le \I_1^+ u (x) \int_{[0, 2\pi)\times [0, \pi]^{k_i-2} }  f(\theta_1, \dots, \theta_{k_i}) \, d \theta_{k_i-1} \dots d \theta_1 =k_i \omega_{k_i} \I_1^+ u (x)
\end{align*}
where 
\[ f(\theta_1, \dots, \theta_{k_i})= \sin^{k_i-2}(\theta_1) \sin^{k_i-3}(\theta_2) \dots \sin(\theta_{k_i-2}). \]
This leads to  \eqref{bound3} and 
\[ \sum_{i=1}^\ell \JJ_{V_{k_i}} v+ \sum_{i=1}^\ell k_i \omega_{k_i} \mu v\leq0 \quad\text{in $\Omega$} \]
 for any choice of the spaces $V_{k_i}$, thus 
 \[ \bar\mu^+_1\leq \left(\sum_{i=1}^\ell k_i \omega_{k_i} \right)^{-1} \bar\mu^+_{k_1, \dots, k_\ell}. \] 
 We also point out that $\bar\mu^+_{k_1, \dots, k_\ell} \le \bar\mu^-_{k_1, \dots, k_\ell}$ for any $k_1, \dots, k_\ell$. As for the bound from below, we observe that  $\K_1^+=\I_1^+$, and refer to \cite{BGS}.
\end{proof}

In uniformly convex domains
$\bar \mu_{k_1, \dots, k_\ell}^\pm=\mu_{k_1, \dots, k_\ell}^\pm$, and this value is the optimal threshold for the validity of the maximum principle. We omit the proof of the next results, as it is completely analogous to Lemma 6.5, Lemma 6.7 and Theorem 6.8 in \cite{BGS}.

\begin{lemma}\label{barrier}
Let $m$ be a positive constant and let $u$ be a solution of
\[ \begin{cases}
\K_{k_1, \dots, k_\ell}^\pm u (x) \ge -m &\text{ in } \Omega\\
u \le 0 &\text{ in } \R^N\setminus \Omega, 
\end{cases} \] 
where the domain $\Omega$ is  uniformly convex. Then there exists a positive constant $C=C(\Omega, m, s)$ such that 
\begin{equation}\label{2406eq12} u(x) \le C \, d(x)^s \end{equation}
for any $x \in \overline \Omega$.
\end{lemma}

\begin{lemma}\label{mu=bar mu}
Let $\Omega$ be a convex domain. Then $\mu_{k_1, \dots, k_\ell}^\pm=\bar \mu_{k_1, \dots, k_\ell}^\pm$.
\end{lemma}

\begin{theorem}\label{max principle}
Let $\Omega$ be a uniformly convex  domain. The operator
\[ \K^+_{k_1, \dots, k_\ell} + \mu \]
satisfies the maximum principle if and only if $\mu < \mu_{k_1, \dots, k_\ell}^+ < +\infty$, and correspondingly 
\[ \K_{k_1, \dots, k_\ell}^- + \mu \]
satisfies the maximum principle if and only if $\mu < \mu_{k_1, \dots, k_\ell}^- < +\infty$ if $k=N$, and 
for any $\mu\in\mathbb R$ if $k < N$.
\end{theorem}

\appendix
\section{Representation formulas}\label{sec:repres}
We recall that in \cite{BGT} some representation formulas for $\I_k^\pm$ and $\JJ_k^\pm$ are given, and Liouville type theorems are proved. Here, we slightly extend their results to the operators $\K_{k_1, \dots, k_\ell}^-$ with $k <N$. However, it does not seem trivial to adapt the arguments to the full general case, precisely when a competition between different terms in the sum arises. Nonetheless, one has 
\begin{lemma}\label{lem:repres}
Let $k < N$. Assume $u(x)=\tilde g(|x|^2) \in L^\infty(\R^N)$ such that $\tilde g$ is convex. Then for any $x \ne 0$ 
\begin{equation}\label{rappr} \K_{k_1, \dots, k_\ell}^- u(x) = \sum_{i=1}^\ell \JJ_{\langle \eta_j^i \rangle_{j=1}^{k_i} }  u(x) = \frac {\sum_{i=1}^\ell C_{k_i, s}} {C_{1, s}} \; \I_{x^\perp } u(x)  \end{equation}
where $\cup_{i=1}^\ell \{ \eta_j^i \}_{j=1}^{k_i}$ is any orthonormal basis of $\R^k$ such that each $\eta_j^i$ is orthogonal to $x$, whereas $x^\perp$ is any unit vector orthogonal to $x$. 
\end{lemma}
\begin{remark}
In particular, if $k_i=1$ for any $i=1,\dots,\ell$, namely when $\K^-_{1, \dots, 1}=\I_k^-$,  we recover \cite[Theorem 3.4]{BGT}, whereas if $\ell=1$, namely if $\K^-_k= \JJ_k^-$, we get \cite[Proposition 5.1]{BGT}. 
\end{remark}

\begin{proof}
The proof immediately follows recalling that, by \cite[Proposition 5.1, Lemma 3.3]{BGT} for any $V_i = \langle \xi_1^i, \dots, \xi_{k_i}^i \rangle$ 
\[ \JJ_{W_i} u(x) \le \JJ_{V_i} u(x) \]
where $W_i$ is any space of dimension $k_i$ which is orthogonal to $x$. Then
\[ \sum_{i=1}^\ell \JJ_{\langle \eta_j^i \rangle_{j=1}^{k_i} } u(x)  \le \sum_{i=1}^\ell \JJ_{V_i} u(x), \]
where $\cup_{i=1}^\ell \{ \eta_j^i \}_{j=1}^{k_i}$ is any orthonormal basis of $\R^k$ such that each $\eta_j^i$ is orthogonal to $x$. 
Hence
\[ \K_{k_1, \dots, k_\ell}^- u(x) = \sum_{i=1}^\ell \JJ_{\langle \eta_j^i \rangle_{j=1}^{k_i} }  u(x). \]

Moreover, 
we know by \cite{BGT} that for any $V_i = \langle \xi_1^i, \dots, \xi_{k_i}^i \rangle$ , 
\begin{equation}\label{rappr:eq1} \JJ_{V_i} u(x)= \frac{C_{k_i, s}}{2} f_i(\Phi) \end{equation}
with 
\[f_i(\Phi)=\int_{\R^{k_i}} \frac{\tilde g (|x|^2+|\tau|^2+2|x|\tau_1 \sin(\Phi)) + \tilde g(|x|^2+|\tau|^2 - 2|x| \tau_1 \sin(\Phi)) - 2 \tilde g(|x|^2) }{\left( \sum_{j=1}^{k_i} \tau_j^2 \right)^{\frac{k_i+2s}{2}}} \, d\tau \]
and $\Phi$ is the angle between $x$ and $\hat \xi$, the unit vector orthogonal to $V_i$ in the $k_i+1$ dimensional space generated by $\xi_1^i, \dots, \xi_{k_i}^i$ and $x$, see \cite[Proposition 5.1]{BGT}. 
Also,
\begin{equation}\label{rappr:eq2} f_i(0)=2 \int_0^{+\infty} (\tilde g(|x|^2+r^2)-\tilde g(|x|^2)) r^{-1-2s}, \end{equation}
which in particular means that it does not depend on $i$.

We now recall \eqref {rappr:eq1} and \eqref{rappr:eq2} to conclude that for any $i=1,\dots,\ell$ and any $\cup_{i=1}^\ell \{ \eta_j^i \}_{j=1}^{k_i}$ orthonormal basis of $\R^k$ such that each $\eta_j^i$ is orthogonal to $x$, 
\[  \JJ_{\langle \eta_j^i \rangle_{j=1}^{k_i} } u(x) = \frac{C_{k_i, s}}{2} f_i(0)= \frac{C_{k_i, s}}{C_{1, s}} \I_{x^\perp} u(x). \]
This yields \eqref{rappr}. 
\end{proof}

As a corollary, we get, following the same arguments as in \cite[Theorem 4.7]{BGT}, the next Liouville type result. 
\begin{corollary}
Assume $k < N$. Then, for any $p \ge 1$ there exist positive classical solutions of the equation
\[ \K_{k_1, \dots, k_\ell}^- u(x) + u^p(x)=0 \text{ in } \R^N. \]
If $p \in (0, 1)$, then there exist nonnegative viscosity solutions $u \not \equiv 0$. 
\end{corollary}
Precisely, 
\[ u(x)=\frac{\alpha}{(a^2 + |x|^2)^{\frac{s}{p-1}}} \]
is a solution for the case $p>1$, for any $a \ne 0$ and for a suitable $\alpha$. On the other hand, a solution if $p=1$ is given  by
\[ u(x)=be^{-\beta |x|^2} \]
with $b >0$ and $\beta$ suitably chosen. Finally, the case $p\in (0,1)$ can be treated exploiting the function
\[ u(x)=\alpha(R^2-|x|^2)^{\frac{s}{1-p}}_+ \]
for a suitable $\alpha>0$ and for any $R>0$.

\section*{Acknowledgements} 
I wish to thank Isabeau Birindelli and Giulio Galise for many fruitful discussions, suggestions, and comments.
I also gratefully acknowledge the financial support from the Portuguese government through FCT - Funda\c c\~ao para a Ci\^encia e a Tecnologia, I.P., under the projects UID/MAT/04459/2020 and PTDC/MAT-PUR/1788/2020 and, when eligible, by COMPETE 2020 FEDER funds, under the Scientific Employment Stimulus - Individual Call (CEEC Individual) - reference number 2020.02540.CEECIND/CP1587/CT0008.

\end{document}